\documentclass[11pt]{amsart}
\usepackage{amssymb, amstext, amscd, amsmath, amssymb}
\usepackage{mathtools, paralist, color, dsfont, rotating, bbm}
\usepackage{verbatim}
\usepackage{enumerate,enumitem}
\usepackage{relsize} 
\usepackage{float}
\usepackage{setspace}

\usepackage{tikz}

\floatstyle{boxed} 
\restylefloat{figure}

\numberwithin{equation}{section}
\usepackage{quoting}
\quotingsetup{vskip=.1in}
\quotingsetup{leftmargin=.17in}
\quotingsetup{rightmargin=.17in}

\let\OLDthebibliography\thebibliography
\renewcommand\thebibliography[1]{
  \OLDthebibliography{#1}
  \setlength{\parskip}{0pt}
  \setlength{\itemsep}{2pt plus 0.5ex}
}

\usepackage[a4paper]{geometry}
\makeatletter
\def\@cite#1#2{{\m@th\upshape\bfseries%
[{#1\if@tempswa{\m@th\upshape\mdseries, #2}\fi}]}}
\makeatother
\theoremstyle{plain}
\newtheorem{theorem}{Theorem}[section]

\theoremstyle{definition}
\newtheorem{definition}[theorem]{Definition}
\newtheorem{example}[theorem]{Example}

\theoremstyle{remark}

\mathtoolsset{centercolon}
\renewcommand{\H}{{\mathcal{H}}}
  
  \newcommand{\J}{{\mathcal{J}}}

\renewcommand{\P}{{\mathcal{P}}}

\renewcommand{\S}{{\mathcal{S}}}
  \newcommand{\T}{{\mathcal{T}}}
  \newcommand{\U}{{\mathcal{U}}}

\newcommand{\ca}{\mathrm{C}^*}

\newcommand{\wot}{\textsc{wot}}

\newcommand{\alg}{\operatorname{alg}}

\newcommand{\sca}[1]{\left\langle#1\right\rangle} 
\newcommand{\nor}[1]{\left\Vert #1\right\Vert} 

\newtheorem*{theorem*}{Theorem}
\newtheorem*{corollary*}{Corollary}
\newtheorem*{proposition*}{Proposition}
\newtheorem*{lemma*}{Lemma}
\newtheorem*{remark*}{Remark}
\newtheorem*{definition*}{Definition}

\usepackage[normalem]{ulem}
\usepackage{xstring}

\begin{document}

\title
{Nondegenerate actions on Hilbert $\ca$-modules}





\begin{abstract}
Let $X$ be a (right) Hilbert $\ca$-module and let $B$ be a $\ca$-algebra acting on $X$ from the left via adjointable operators. In this note we establish  the equivalence of two notions of nondegeneracy for such an action of $B$ on $X$. Furthermore, we explain how techniques from an earlier paper of Kim can be used to give a short proof of a recent result of Dessi, Kakariadis and Paraskevas.
 \end{abstract}

\maketitle




\section{Introduction}

The purpose of this short note is twofold. First we will explain how the work of Kim \cite{Kim} leads to a quick proof of a recent result of Dessi, Kakariadis and Paraskevas  \cite[Theorem 4.4 ((vi) $\Rightarrow$ (viii)]{DKP}, which answered a question from \cite{KR2}. Secondly, we will elaborate on Kim's arguments with the purpose of providing a new result in the theory of Hilbert $\ca$-modules, Theorem~\ref{main;cor}, regarding the non-degeneracy of left actions on Hilbert $\ca$-modules via adjointable operators. This is a more general result than that of Dessi, Kakariadis and Paraskevas and the focal result of this note.

In 2018, the authors wrote a paper \cite{KR1} regarding the Hao-Ng problem. One of the key results in that paper was a condition sufficient for the hyperrigidity of a tensor algebra of a $\ca$-correspondence $(X, A)$: the left action of Katsura's ideal $\J_X$ on $X$ should be nondegenerate, i.e., $[\J_XX] = X$. (The original statement of the result required $X$ to be countably generated but the proof worked equally well for arbitrary $\ca$-correspondences, as noted in later iterations and the final version of \cite{KR1}.) Note that the nondegeneracy of the left action of $\J_X$ automatically implies that the $\ca$-correspondence $(X, A)$ is nondegenerate.

On May 24, 2019, two papers \cite{KR2, Kim} appeared independently on the arXiv addressing the converse of the above result (both are now published). In one of these two papers \cite{Kim}, Kim proved the action of Katsura's ideal being nondegenerate is also a necessary condition for the tensor algebra of $X$ to be hyperrigid. Thus, combining Kim's result with our earlier result in \cite{KR1}, the following complete characterization was obtained. 

\begin{theorem}[Katsoulis and Ramsey, Kim] Let $(X, A)$ be a nondegenerate \\ $\ca$-correspondence. Then the tensor algebra $\T^+(X)$ is hyperrigid if and only the left action of Katsura's ideal $\J_X$ on $X$ is nondegenerate, i.e. $[ \J_X X ]=X$.
\end{theorem}

The other paper \cite{KR2}, written by the authors, contained a potentially weaker result on $\sigma$-degeneracy than that of Kim in \cite{Kim} on degeneracy.

\begin{definition}
Let $X$ be a $\ca$-correspondence over a $\ca$-algebra $A$ and let $\sigma: A \rightarrow B(\H)$ be a representation of $A$ on Hilbert space $\H$. If $C$ is a $\ca$-subalgebra of $A$, then the left action of $C$ on $X$ is said to be $\sigma$-nondegenerate if and only if $ CX \otimes_{\sigma}\H = X \otimes_{\sigma}\H $. (Otherwise, we say that $C$ acts $\sigma$-degenerately on $X$.)
\end{definition}

Specifically, the authors showed in  \cite{KR2} that if $\J_X$ acts $\sigma$-degenerately on $X$ then $\T^+(X)$ cannot be hyperrigid. In other words for $\T^+(X)$ to be hyperrigid, the left action of $\J_X$ on $X$ must be spatially nondegenerate on $X$ on all of its Hilbert space representations. This seems to be a weaker condition than the nondegeneracy of the left action of $\J_X$ on $X$. Indeed the nondegeneracy of the left action of $\J_X$ on $X$ forces $\sigma$-nondegenerancy but it was not clear what happens the other way around. Nevertheless, it was shown in \cite{KR2} that the two notions of nondegeneracy are equivalent for $\ca$-correspondences of topological graphs and most importantly, a condition on the graph was given of when $\J_X$ acts nondegenerately on $X$ for such correspondences. The general case was left as an open question in \cite{KR2}. Recently, Dessi, Kakariadis and Paraskevas \cite{DKP} verified that these two notions of non-degeneracy are indeed equivalent for \textit{Katsura's ideal} and all $\ca$-correspondences, thus resolving the open question from \cite{KR2}. As promised earlier, we will give a very short proof of the Dessi, Kakariadis and Paraskevas and actually more.

\section{The main results}

Let $A$ be a $\ca$-algebra, then an \emph{inner-product right $A$-module} is a linear space $X$ which is a right $A$-module together with a map
 \begin{align*}
\sca{\cdot,\cdot}\colon X \times X \rightarrow A
\end{align*}
which satisfies the properties of an $A$-valued inner product with linearity holding for the second variable. (See \cite{Lan} for precise definitions.)
For $\xi \in X$ we write $\nor{\xi}^2 := \nor{\sca{\xi,\xi}}$ and one can deduce that $\nor{\cdot}$ is actually a norm. The space $X$ equipped with that norm will be called an \emph{ Hilbert $A$-module} if it is complete. For a Hilbert $A$-module $X$ we define the set $L(X)$ of the \emph{adjointable maps} that consists of all maps $s:X \rightarrow X$ for which there is a map $s^*: X \rightarrow X$ such that $\sca{s\xi,\eta}= \sca{\xi,s^*\eta}$, $\xi, \eta \in X$.

A $\ca$-correspondence $(X, A)$ consists of a Hilbert A-module $(X, A)$ and an action of $A$ on $X$ from the left, i.e., a $*$-homomorphism $\phi: A \longrightarrow L(X)$. (We write $a\xi$ instead of $\phi(a)\xi$, $a \in A$, $\xi \in X$.)
If $(X, A)$ is a $\ca$-correspondence, then a (Toeplitz) representation $(\pi,t)$ of $(X, A)$ on Hilbert space $\H$, is a pair of a $*$-homomorphism $\pi\colon A \rightarrow B(\H)$ and a linear map $t\colon X \rightarrow B(\H)$, such that
\begin{enumerate}
 \item $\pi(a)t(\xi)\pi(a')=t(a\xi a')$,
 \item $t(\xi)^*t(\eta)=\pi(\sca{\xi,\eta})$, 
\end{enumerate}
for all $a, a' \in A$ and $\xi, \eta \in X$.
Finally, if $(X, A)$ is a Hilbert $A$-module and $\sigma: A \rightarrow B(\H)$ a Hilbert space representation, then $X \otimes_{\sigma} \H$ will denote the stabilized tensor product of $X$ and $\H$, i.e., the Hilbert space\footnote{actually, anti-Hilbert space}  generated by the quotient of the vector space tensor product $X \otimes_{\alg} \H$ by the subspace generated by the elements of the form
\begin{align*}
\xi a \otimes h - \xi \otimes \sigma(a) h , \quad \xi, \eta \in X, a
\in A,
\end{align*}
which becomes a pre-Hilbert space when equipped with
\[
\sca{\xi_1\otimes h_1, \xi_2\otimes h_2} := \sca{h_1,
\sigma(\sca{\xi_1,\xi_2}) h_2}, \quad \xi_1,\xi_2 \in X, h_1,h_2\in \H.
\]

One of the key points to this paper is that the degeneracy of (left) action of a $\ca$-algebra on a $\ca$-correspondence can be detected through representations, i.e., a degenerate action leads to $\pi$-degeneracy for a some representation $\pi$. In spite of this, a degenerate action may also lead to $\sigma$-nondegeneracy for some other representation $\sigma$, as the following example shows.

\begin{example}
Let $\H=\ell^2(\mathbb N)$ and  $(B(\H),B(\H))$ be the C*-correspondence with the inner product $\langle a,b\rangle = a^*b$. The left action of $K(\H)$ is degenerate, $[K(\H)B(\H)] = K(\H) \neq B(\H)$. Now consider the identity representation then we claim that $$K(\H)\otimes_{\operatorname{id}} \H = B(\H) \otimes_{\operatorname{id}} \H\,.$$
This follows because for any $a\in B(\H)$ and $h\in \H$, if $P_n$ is the projection onto the subspace spanned by $\{e_1, \dots, e_n\}$, then 
\begin{align*}
\|a\otimes h - aP\otimes h\| & = \langle (a-aP_n)\otimes h, (a-aP_n)\otimes h\rangle^{1/2} 
\\ & = \langle h, (a-aP_n)^*(a-aP_n)h\rangle^{1/2}
\\ & =  \| (a-aP_n)h\| \xrightarrow{n \rightarrow \infty} 0\,.
\end{align*}
Hence, the left action of $K(\H)$ is id-nondegenerate.

To get $\pi$-degeneracy we must turn to the only other non-zero irreducible representation, $\pi : B(\H) \rightarrow B(\H)/K(\H)$. Then the left action of $K(\H)$ on $B(\H)$ is $(\operatorname{id} \oplus \pi)$-degenerate. As we will see, degeneracy of the left action always leads to some representation $\pi$ that gives $\pi$-degeneracy.
\end{example}

Our first task consists of proving the Dessi, Kakariadis and Paraskevas result \cite[Theorem 4.4]{DKP} using Kim's ultrafilter argument. For the benefit of the reader acquainted with \cite{Kim}, we retain Kim's notation from the proof of \cite[Theorem 3.5]{Kim}. All other readers are directed to the more general Theorem~\ref{elem} below for which we give all pertinent definitions and a complete proof. 

\begin{theorem}[Theorem 4.4 ((vi) $\Rightarrow$ (viii)) Dessi, Kakariadis and Paraskevas \cite{DKP}]  \label{DKPthm} Let $X$ be a nondegenerate $\ca$-correspondence over a $\ca$-algebra $A$. If the left action of $\J_X$ on $X$ is degenerate, then there exists a representation $\sigma: A\rightarrow B(\H)$ so that the left action of $\J_X$ on $X$ is $\sigma$-degenerate.
\end{theorem}

\begin{proof}
Assuming that the left action of $\J_X$ on $X$ is degenerate  \cite[pg. 12, line 14]{Kim}, Kim identifies in  \cite[pg. 14, line 13]{Kim} a representation 
\[
(\overline{\pi}^0, \overline{\pi}^1)\colon (A, X)\longrightarrow B(H^{\U}),
\]
and an element $x \in X$, so that $P \overline{\pi}^1(x)\neq \overline{\pi}^1(x)$, where $P$ is the support projection of $ \overline{\pi}^0(\J_X)$. This implies that $(I-P)\overline{\pi}^1(X)(H^{\U})\neq\{0\}$. Hence $\overline{\pi}^1(X)(H^{\U})\nsubseteq \overline{\pi}^0(\J_X)(H^{\U})$. Therefore, with the aid of the unitary operator 
\[
U: X\otimes_{\overline{\pi}^0} H^{\U} \longrightarrow [\overline{\pi}^1 (X)H^{\U}]; \xi \otimes k \longmapsto \overline{\pi}^1 (\xi)k
\]
we see that $\J_X$ acts $\overline{\pi}^0$-degenerately on $X$, or otherwise 
\[
[\overline{\pi}^1 (X)H^{\U}] = [\overline{\pi}^1 (\J_XX)H^{\U}] \subseteq [\overline{\pi}^0 (\J_X)H^{\U}] ,
\]
which is a contradiction. This completes the proof.
\end{proof}

We now revisit Kim's arguments from \cite[Theorem 3.4]{Kim} with the goal of providing a new result in the theory of Hilbert $\ca$-modules, i.e., Theorem~\ref{main;cor}. From this result (or Theorem~\ref{elem}) one can easily deduce Theorem~\ref{DKPthm}, since $\J_X$ is just an ideal of the ambient $\ca$-algebra $A$. This is the only information one needs to know about $\J_X$ and so we will skip its definition, or that of a tensor algebra. All other definitions have been given earlier.

\begin{theorem} \label{elem} Let $X$ be a $\ca$-correspondence over a $\ca$-algebra $A$ and let $C$ be a $\ca$-subalgebra of $A$. If the left action of $C$ on $X$ is degenerate, then there exists a representation $(\pi, t)$ of $(X, A)$ on Hilbert space $\H$ so that the left action of $C$ on $X$ is $\pi$-degenerate.
\end{theorem}

\begin{proof}
Assume that the $\ca$-correspondence $(X, A)$ is faithfully represented on a Hilbert space $\H_0$. The collection $\S$ of all positive contractions in $C$ with the usual partial ordering of positive operators forms an approximate unit for $C$. (If $C$ is unital, then simply take $\S=\{1\}$.) 

Since the left action of $C$ on $X$ is degenerate, there exists $x \in X$ and $\epsilon > 0$ so that $\| x -cx\| > \sqrt{\epsilon}$, for all $c \in \S$. Hence 
\begin{equation} \label{first}
\| (1-c)^{1/2}x\| > \sqrt{\epsilon}, \quad \mbox{for all } c \in \S,
\end{equation}
or otherwise $\| x -cx\|\leq  \| (1-c)^{1/2}\| \| (1-c)^{1/2}x\|\leq \sqrt{\epsilon}$, a contradiction. Now (\ref{first}) implies that for each $c \in \S$, there exists a unit vector $h_c\in \H_0$ so that 
\begin{equation} \label{second}
\langle (x^*x - x^*cx)h_c\mid h_c\rangle = \| (1-c)^{1/2}xh_c\|^2 > \epsilon.
\end{equation}
Let $\P(\S)$ be the powerset of $\S$
and let $\U$ be a (necessarily non-principal) ultrafilter in $\P(\S)$ containing all sets of the form $ \{c \in \S \mid c \geq b\}$, $b \in \S$. Let $\H:= \prod_{\U} \H_0$ be the ultraproduct of $\H_0$ along the ultrafilter $\U$. If $(k_c)_{c \in \S} \in \H$, then we define 
\begin{align*}
\pi(a)(k_c)_{c \in \S}&:= (ak_c)_{c \in \S}, \quad a \in A \\
t(x)(k_c)_{c \in \S} &:= (xk_c)_{c \in \S}, \quad x \in X.
\end{align*}
It is easy to see that $(\pi, t)$ is a representation of $(X, A)$. 

Let $P := \wot-\lim_{b \in \S}\pi(b)$ be the support projection of $\pi(C)$ on $\H$. For the element $x \in X$ from (\ref{second}), we have the following
  
\vspace{.1in} 
\noindent \textbf{Claim:} $ \lim_{c, \U}\{ \langle(x^*x - x^*bx)h_c\mid h_c\rangle\}_{c \in \S} \geq \epsilon$, for all $ b \in \S$.

\vspace{.07in}

\noindent By way of contradiction assume otherwise for a specific $b \in \S$. By the very definition of convergence along an ultrafilter, we have that
\[
S_b :=\{ c \in \S\mid \langle ( x^*x - x^*bx) h_c\mid h_c\rangle  < \epsilon\} \in \U.
\]
On the other hand, (\ref{second}) implies that $b \in \S_b^c$. Actually, the inequalities
\[
\langle (x^*x - x^*bx) h_c\mid h_c\rangle \geq  \langle (x^*x - x^*cx)h_c\mid h_c\rangle\ > \  \epsilon , \quad c\geq b 
\]
imply that  
\[
\U \ni \{ c \in \S \mid c \geq b\} \subseteq \S_b^c.
\]
Therefore $\S_b^c$ also belongs to $\U$, which is absurd. This proves the claim.

\vspace{0.1in}

The claim now implies that for $h :=(h_c)_{c \in \S} \in \H$, we have 
\[
\Big\langle\big(t(x)^*t(x) - t(x)^*\pi(b)t(x)\big)h \mid h \Big\rangle\geq \epsilon
\]
and by taking weak limits,
\[
\| (I-P)t(x)h\|^2 =\Big\langle\big(t(x)^*t(x) - t(x)^*P t(x)\big)h \mid h \Big\rangle\geq \epsilon.
\]
Hence $(I-P)t(X)\H \neq \{0\}$ and so $[t(X)\H] \nsubseteq [\pi(C)\H]$. Therefore, with the aid of the unitary operator 
\[
U: X\otimes_{\pi} \H \longrightarrow [t(X)\H]; \xi \otimes k \longmapsto t(\xi)k
\]
we see that $C$ acts $\pi$-degenerately on $X$, or otherwise 
\[
[t(X)\H] = [t(CX)\H] \subseteq [\pi(C)\H],
\]
which is a contradiction. This completes the proof.
\end{proof}

The previous theorem leads to the following new criterion for the nondegeneracy of the left action of a $\ca$-algebra on a Hilbert $\ca$-module. 

\begin{theorem} \label{main;cor}
Let $A, B$ be $\ca$-algebras, let $X$ be a Hilbert $\ca$-module over $A$ and assume that $B$ is acting from the left via adjointable operators.  Then $B$ acts nongenerately on $X$ if and only if for any representation $\sigma: A \rightarrow B(\H)$, the $\ca$-algebra $B$ acts $\sigma$-nondegenerately on $X$, i.e., $ BX \otimes_{\sigma}\H = X \otimes_{\sigma}\H $.
\end{theorem}
\begin{proof}
One direction is immediate: if $B$ acts nondegenerately, then it also acts $\sigma$-nondegenerately. For the other direction, assume that $B$ acts degenerately on $X$. Consider $X$ now as a $\ca$-correspondence over $A\oplus B$ with actions
\[
(a\oplus b) \xi (a'\oplus b'):= b\xi a'
\]
and inner product
\[
[\xi, \zeta]:=\langle \xi, \zeta\rangle \oplus 0, 
\]
where $a, a'\in A$,  $b , b'\in B$ and $\xi, \zeta \in X$.  Since $A\oplus B$ acts degenerately on $X$, use the previous result to find a representation $\pi : A\oplus B \rightarrow B(\H)$ so that 
\[
(A\oplus B)X \otimes_{\pi}\H \neq X \otimes_{\pi}\H.
\]
If $\sigma :=\pi\mid_{A \oplus 0}$, then we have from above that $ BX \otimes_{\sigma}\H \neq X \otimes_{\sigma}\H $ and the conclusion follows.
\end{proof}

\section*{Acknowledgements} Elias Katsoulis was partially supported by the NSF grant DMS-2054781. Christopher Ramsey was supported by the NSERC Discovery Grant 2019-05430.


\begin{thebibliography}{99}


\bibitem{DKP} J. Dessi, E. Kakariadis and A. Paraskevas, \textit{On hyperrigidity and non-degenerate C*-correspondences}, Studia Math. \textbf{287} (2026), 257--303

\bibitem{KR1} E. Katsoulis and C. Ramsey, \textit{The non-selfadjoint approach to the Hao-Ng isomorphism problem}, International Mathematics Research Notices (IMRN) (2021), 1160-1197.

 \bibitem{KR2} E. Katsoulis and C. Ramsey, \textit{The hyperrigidity of tensor algebras of  $\ca$--correspondences}, J. Math. Anal. Appl. \textbf{483} (2020), no. 1, 123611.
 
\bibitem{Kim} S.-J. Kim, \textit{Hyperrigidity of $\ca$--correspondences} Integral Equ. Oper. Theory \textbf{93} (4) (2021) 47

\bibitem{Lan}  E.C. Lance, \textit{Hilbert $\ca$-modules}, London Mathematical Society Lecture Note Series, 210. Cambridge University Press, Cambridge, 1995. x+130 pp. ISBN: 0-521-47910-X

\end{thebibliography}
\end{document}